\documentclass[11pt,a4paper,reqno]{amsart}
\usepackage{amsfonts}
\usepackage{graphicx}   
\usepackage{graphics}
\usepackage{epsfig}

\usepackage{anysize}
\marginsize{3cm}{3cm}{3cm}{3cm}

\newtheorem{theorem}{Theorem}[section]

\makeatletter

\begin{document}

\setcounter{page}{1}
\begin{flushleft}
\end{flushleft}
\bigskip

\title[\centerline{Zonal and Associated Functions on $SO_{0}(p,q)$ Groups
\hspace{0.5cm}}] {Zonal and Associated Functions on $SO_{0}(p,q)$ Groups}

\author[\hspace{0.7cm}\centerline{B.A.Rajabov}] {B.A. RAJABOV$^1$}

\thanks {\noindent $^1$ N. Tusi Shamakhi Astrophysics Observatory, National Academy of Sciences of Azerbaijan, AZ2243, Shamakhi, Azerbaijan,\\
\indent \,\,\, e-mail: balaali.rajabov@mail.ru;}


\bigskip
\begin{abstract}
Explicit expressions for associated spherical functions of $SO(p,q)$ matrix groups are obtained using a generalized hypergeometric series of two variables. In this paper we present explicit expressions for zonal functions of de Sitter groups and the group of conformal invariance. Moreover, we present a theorem on the transformation of derivative of distributions, concentrated on smooth surfaces, with respect to infinite-dimensional Lie group $C^{\infty}\left(\mathbb{R^{\mathit{n}}},GL(n)\right)$.


\end{abstract}

\maketitle
\bigskip


\section{Introduction}
The $SO(p,q)$ groups of pseudo-orthogonal matrices and their representations are broadly used in different fields of physics, particularly in quantum field theory, high- energy physics, cosmology and solid-state physics [1-2].

Matrix elements of irreducible unitary representations (IUR) of these groups, particularly spherical functions, play an essential role in the theory of representations of $SO(p,q)$ groups. The detailed study of these functions in the case of $SO(p,1)$ groups can be found in [3,4]. Unlike them, the construction of spherical functions of $SO(p,q),p\geq q\geq2$ groups has not yet been completed. The paper [2] by N. Ya. Vilenkin and A.P. Pavluk, where the relationship between the spherical functions of matrix groups and Herz's functions of matrix arguments [5] was established, has attracted a
remarkable attention.

The purpose of this paper is to find the spherical functions of $SO(p,q),p\geq q\geq2$, groups. The main result consists of the fact that associated spherical functions of the groups $SO(p,q),p\geq q\geq2$, are expressed as generalized hypergeometric Horn's functions of two variables by a unique formula, in spite of essentially different forms of the integral representations of spherical functions at $p\geq q\geq 3$ and $p\geq2, q=2$. The preliminary statements about the obtained results were made in the Proceedings of the Azerbaijan National Academy of Sciences [6-7].

Furthermore, the obtained formula has been proved to be valid for $q=1$, i.e., for $SO(p,1), p\geq2$ groups. This work is continuation of the publication [8]. Main results from this paper are available on website arXiv.org, [9-10]. In this publication we reduce rectified expressions for some formulas [8-10], and also new expressions for zonal functions.

\section{Most degenerate irreducible unitary representations of $SO_{0}(p,q)$ groups}
The most degenerate representations of groups $SO_{0}(p,q)$, i.e., connected components of units of groups of motions of $(p + q)$ -dimensional vector space holding invariant quadratic form:
\[ [k,k]=k_{1}^{2}+\ldots+k_{q}^{2}-k_{q}^{2}-\ldots k_{p+q}^{2}, \]
are given by a complex number $\sigma$ and a number $\varepsilon$, which take the values 0 and 1, and are constructed in the space of homogeneous functions, $F\left(\cdot\right)$, of given parity and defined on the cone $\left[k,k\right]=0$, [1--6].

Let us denote by $D^{(\sigma,\varepsilon)}$ the space of infinitely differentiable functions, $F\left(\cdot\right)$, defined on the cone $\left[k,k\right]=0$ without the point k = 0 and satisfying the following condition:
\begin{equation}\
F(ak)=\left|a\right|^{\sigma}F(k)\mathrm{sign}^{\varepsilon}a,\quad a\neq0,\quad\varepsilon=0,\,1.
\end{equation}
The action of the operator$T^{(\sigma,\varepsilon)}(g)$ of $SO_{0}(p,q)$ in the space  $D^{(\sigma,\varepsilon)}$ is determined as follows:
\begin{equation}
T^{(\sigma,\varepsilon)}(g)F(k)=F\left(g^{-1}k\right),\quad g\in SO_{0}(p,q).
\end{equation}
The space of the representation and the representation itself may have different realizations. Let us consider one of them.

Let us introduce the spherical system of coordinates on the cone $\left[k,k\right]=0$:
\begin{equation}
k=\omega(\cos\phi,\boldsymbol{\vec{\eta}}\sin\phi,\cos\chi,\boldsymbol{\vec{\xi}}\sin\chi),
\end{equation}

where $0<\omega<\infty$, and $\boldsymbol{\vec{\eta}}$ and $\boldsymbol{\vec{\xi}}$ are $(q-1)$-dimensional and $(p-1)$-dimensional unit vectors, respectively.

Spherical angles $\phi$ and $\chi$ change within the following limits:

$0\leq\phi\leq\pi,\;0\leq\chi\leq\pi,$ if $p\geq q\geq3$;

$0\leq\phi<2\pi,\;0\leq\chi\leq\pi,$ if $p=3,\:q=2$ (in this case
$\boldsymbol{\vec{\eta}}$ is reduced to a constant value which we take as $\boldsymbol{\vec{\eta}}=1$);

$0\leq\phi<2\pi,\;0\leq\chi<2\pi,$ if $p=q=2$ (in this case $\boldsymbol{\vec{\xi}}$
and $\boldsymbol{\vec{\eta}}$ are reduced to a constant values which we take as $\boldsymbol{\vec{\eta}}=\boldsymbol{\vec{\xi}}=1$). 

Let us consider the restriction of functions from $D^{(\sigma,\varepsilon)}$ on the cross-section, $\omega=1$ of the cone $\left[k,k\right]=0$:
\begin{equation}
f(\phi,\boldsymbol{\vec{\eta}};\chi,\boldsymbol{\vec{\xi}}\;)=F(k)|_{\omega=1}.
\end{equation}
Then, in virtue of homogeneity of $SO_{0}(p,q)$ we obtain:
\begin{equation}
F(k)=\omega^{\sigma}f(\phi,\boldsymbol{\vec{\eta}};\chi,\boldsymbol{\vec{\xi}}\;).
\end{equation}
From equations (1)-(4) it is evident that function $f\left(\cdot\right)$ also has the given parity $\varepsilon$:
\[
f(\phi,\boldsymbol{\vec{\eta}};\chi,\boldsymbol{\vec{\xi}}\;)=(-1)^{\varepsilon}f(\pi-\phi,-\boldsymbol{\vec{\eta}};\pi-\chi,-\boldsymbol{\vec{\xi}}\;).
\]
The mappings (4)-(5) establish a one-to-one correspondence between $D^{(\sigma,\varepsilon)}$ and the space of infinitely differentiable functions defined on $S^{q}\otimes S^{p}$. From (2)-(5), using the same notation for the space and the operator ($D^{(\sigma,\varepsilon)}$ and $T^{(\sigma,\varepsilon)}$, respectively) of the representations, we obtain the following expression:
\begin{equation}
T^{(\sigma,\varepsilon)}(g)f(\phi,\boldsymbol{\vec{\eta}};\chi,\boldsymbol{\vec{\xi}}\;)=\left(\omega_{g}/\omega\right)^{\sigma}f(\phi_{g},\boldsymbol{\vec{\eta}}_{g};\chi_{g},\boldsymbol{\vec{\xi}}_{g}),
\end{equation}
where $\omega_{g},\phi_{g},\boldsymbol{\vec{\eta}}_{g};\chi_{g},\boldsymbol{\vec{\xi}}_{g}$ are found from the relation of: $g^{-1}k=k_{g}$. In the space $D^{(\sigma,\varepsilon)}$, let us introduce the scalar product:\footnote{The bar denotes complex conjugation.}
\begin{equation}
(f_{1},f_{2})=\int f_{1}(\phi,\boldsymbol{\vec{\eta}};\chi,\boldsymbol{\vec{\xi}}\;)\overline{f_{2}(\phi,\boldsymbol{\vec{\eta}};\chi,\boldsymbol{\vec{\xi}}\;)}(\sin\phi)^{q-2}d\phi(\sin\chi)^{p-2}d\chi(d\boldsymbol{\vec{\eta}})(d\boldsymbol{\vec{\xi}}\;),
\end{equation}
where $(d\boldsymbol{\vec{\eta}})$ and $(d\boldsymbol{\vec{\xi}}\;)$ are normalized measures on  $S^{q-1}$ and  $S^{p-1}$, respectively, [3].

In the case $q=2$ or $p=q=2$, the differentials $(d\boldsymbol{\vec{\eta}})$ or $(d\boldsymbol{\vec{\eta}})$$(d\boldsymbol{\vec{\xi}}\;)$ are omitted because they are constants, and the other variables are integrated over the whole region.

It follows directly form (6) that the scalar product (7) is invariant for $\mathrm{Re}\,\sigma=-\frac{p+q-2}{2},\;\varepsilon=0,\,1$. By filling the space $D^{(\sigma,\varepsilon)}$ with the scalar product (7), we obtain the Hilbert space $H^{(\sigma,\varepsilon)}$, and the most degenerated IUR's of the continuous principal series of the group $SO(p,q)$. It is easy to show that $(\sigma,\,\varepsilon)$ and $(2-p-q-\sigma,\,\varepsilon)$ representations are unitary equivalent [6].

Formulas (6)--(7) allow the integral representations for matrix elements of the operator of hyperbolic rotations on the surface $(k_{1},k_{p+q})$ in canonical basis to be determined [3]. With respect to the canonical basis vector (which is invariant under the subgroup $SO(p)\otimes SO(q)$) the matrix elements which are the zonal spherical functions of the group $SO(p,q)$ as determined in [3] are of interest. Denoting these functions by $Z_{\sigma}^{[p,q]}(\alpha)$ we obtain from (6)--(7) the following integral representations:
\begin{enumerate}	
	\item In the case of $p\geq q\geq3$;
	\begin{equation}
	Z_{\sigma}^{[p,q]}(\alpha)=\frac{\Gamma\left(\frac{p}{2}\right)\Gamma\left(\frac{q}{2}\right)}{\pi\Gamma\left(\frac{p-1}{2}\right)\Gamma\left(\frac{q-1}{2}\right)}\,\intop_{-1}^{+1}\intop_{-1}^{+1}\,\Lambda^{\sigma/2}\left(1-x^{2}\right)^{\frac{p-3}{2}}\left(1-y^{2}\right)^{\frac{q-3}{2}}dxdy;
	\end{equation}
	
	\item In the case of $p\geq3,\;q=2$;
	\begin{equation}
	Z_{\sigma}^{[p,2]}(\alpha)=\frac{\Gamma\left(\frac{p}{2}\right)}{\pi^{3/2}\Gamma\left(\frac{p-1}{2}\right)}\intop_{0}^{2\pi}\intop_{-1}^{+1}\Lambda^{\sigma/2}\left(1-x^{2}\right)^{\frac{p-3}{2}}d\phi dx;
	\end{equation}
	\item In the case of $p=q=2$;
	\begin{equation}
	Z_{\sigma}^{[2,2]}(\alpha)=\frac{1}{4\pi^{2}}\intop_{0}^{2\pi}\,\intop_{0}^{2\pi}\Lambda^{\sigma/2}d\phi d\chi.
	\end{equation}
	
\end{enumerate}

In (8)--(10) we used the following notation:
\begin{eqnarray}
\Lambda\left(\alpha;x,y\right) & = & (\cos\phi\cosh\alpha-\cos\chi\sinh\alpha)^{2}+\sin^{2}\phi=\nonumber \\
& = & 1+\left(x^{2}+y^{2}\right)\sinh^{2}\alpha-2xy\sinh\alpha\cosh\alpha,\nonumber \\
&  & x=\cos\chi,\quad y=\cos\phi.
\end{eqnarray}

It is to be noted that zonal function exists only for even representations $(\varepsilon=0)$.

\section{Zonal functions and Horn's series}
The main formulas used for the calculation of zonal functions of $SO_{0}(p,q)$ groups are the integral representations (8)--(10) and the Taylor expansions for the function $\Lambda^{\sigma/2}$ given as:
\[
\Lambda^{\sigma/2}=\sum_{\nu=0}^{1}\frac{\left(-\sigma xy\tanh\alpha\right)^{\nu}}{\cosh\alpha}\sum_{l=0}^{\infty}\frac{\left(\nu+\frac{1}{2}\right)_{l}}{l!}\left(\tanh\alpha\right)^{2}\times
\]
\begin{equation}
\times\,F_{2}\left(\nu-\frac{\sigma}{2},-l,-l;\nu+\frac{1}{2},\nu+\frac{1}{2};x^{2},y^{2}\right).
\end{equation}
Here we use the following notation for Pochhammer's symbols,  $(a)_{n}$, and Appell's functions of a second kind [11] :
\[
F_{2}\left(\nu-\frac{\sigma}{2},-l,-l;\nu+\frac{1}{2},\nu+\frac{1}{2};x^{2},y^{2}\right)=\sum_{m,n=0}^{l}\frac{\left(\nu-\sigma/2\right)_{m+n}(-l)_{m}(-l)_{n}}{\left(\nu+1/2\right)_{m}\left(\nu+1/2\right)_{n}m!n!}x^{2m}y^{2n}.
\]
Series (12) converges uniformly and absolutely for sufficiently small values of $\alpha$,  namely at $\left(\cosh(2\alpha)<3\right)$ (the sufficient condition!).

Substituting the expansion into (12), alternatively in (8)--(10), and integrating for zonal function of $SO_{0}(p,q), p > q > 2$, groups, we obtain the following expression:
\begin{eqnarray}
Z_{\sigma}^{[p,q]}(\alpha) & = & \sum_{m=0}^{\infty}\frac{(1/2)_{m}(-\sigma/2)_{m}\left(1-\frac{\sigma+q}{2}\right)_{m}}{\left(p/2\right)_{m}(q/2)_{m}m!}\,\frac{\left(\tanh\alpha\right)^{2m}}{\cosh\alpha}\times\nonumber \\
& \times & _{3}F_{2}\begin{pmatrix}-m, & 1-m-\frac{p}{2}, & \frac{\sigma+q}{2};\\
&  &  & 1\\
1-m+\frac{\sigma}{2}, & \frac{\sigma+q}{2}-m;
\end{pmatrix}.
\end{eqnarray}

From the results of Horn's theory for hypergeometric series of two variables, it follows that the series (13) converges for all finite $\alpha$.

Furthermore, assuming $q=1$ in (13), after necessary simplifications, we obtain the following formula:
\begin{equation} Z_{\sigma}^{[p,1]}(\alpha)=\,_{2}F_{1}\left(-\frac{\sigma}{2},\frac{\:1-\sigma}{2};\:\frac{p}{2};\:\tanh^{2}\alpha\right)(\cosh\alpha)^{\sigma},\
\end{equation}
which exactly coincides with the expression for zonal functions of $SO_{0}(p,1)$ groups. Thus, formula (13) is valid for all $SO_{0}(p, q), p>2, q>1$ groups. 

Expression (13) for zonal functions of $SO_{0}(p, q)$ groups can be rewritten more compactly with generalized hyperge­ometric functions of two variables, /See Appendix, Eq.(51)/:
\begin{equation}
Z_{\sigma}^{[p,q]}(\alpha)=\frac{1}{\cosh\alpha}\,_{4}F_{2}\left[\begin{array}{c|cc|c}
10 &-\sigma$/2$, & 1-\frac{\sigma+q}{2};&\\
01 &(\sigma+$q)/2$, & &\\
11 & $1/2$, & &tanh^{2}\alpha,tanh^{2}\alpha\\\cline{1-3}
11 & $q/2$, & &\\
10 & $p/2$, & &\\
\end{array}\right]
\end{equation}

Rearranging indexes of toting it is possible to receive alternative expressions for zonal functions:
\begin{equation}
Z_{\sigma}^{[p,q]}(\alpha)=\frac{1}{\cosh\alpha}\,_{4}F_{2}\left[\begin{array}{c|cc|c}
10 &-\sigma$/2$, & 1-\frac{\sigma+p}{2};&\\
01 &(\sigma+$p)/2$, & &\\
11 & $1/2$, & &tanh^{2}\alpha,tanh^{2}\alpha\\\cline{1-3}
11 & $p/2$, & &\\
10 & $q/2$, & &\\
\end{array}\right]
\end{equation}

The formulas (15)--(16) express properties of a symmetry of zonal functions concerning permutation $p$ and $q$.

The above-presented results can be summarized in terms of the following theorem:

\begin{theorem}[]
	Zonal spherical functions for all $SO(p,q),p\geq 2, q\geq 1$ groups in general are expressed by hypergeometric functions of two variables according to the formulas (15)-(16).
\end{theorem}

\begin{proof}
	In order to complete the proof of the theorem it is sufficient to note that the zonal functions (8)-(10) are analytic functions of $\alpha$ and sufficient to use the principle of monodromy.
\end{proof}
\section{The canonical basis of IUR's of group $SO_{0}(p,q)$}
The canonical basis of most degenerate IUR's of $SO_{0}(p,q)$ groups is constructed using the results of [3]. It follows that the elements of canonical basis can be represented in the following form:
\begin{enumerate}
	\item In the case of $p\geq q\geq3$:
	\begin{eqnarray}
	\Xi_{\lambda lL,\mu mM}^{(\sigma,\varepsilon)}(\phi,\boldsymbol{\vec{\eta}};\chi,\boldsymbol{\vec{\xi}}\;) & = & a_{\lambda l\mu m}^{pq}C_{\lambda-l}^{l+\frac{q-2}{2}}(\cos\phi)\sin^{l}\phi\Xi_{L}^{l}(\boldsymbol{\vec{\eta}})\times\nonumber \\
	& \times & C_{\mu-m}^{m+\frac{p-2}{2}}(\cos\chi)\sin^{m}\chi\Xi_{M}^{m}(\boldsymbol{\vec{\xi}}\;),\nonumber \\
	&  & \lambda\geq l\geq0,\quad\mu\geq m\geq0;
	\end{eqnarray}
	\item In the case of $p\geq3,\;q=2$:
	\begin{eqnarray}
	\Xi_{\lambda\mu mM}^{(\sigma,\varepsilon)}(\phi;\chi,\boldsymbol{\vec{\xi}}\;) & = & a_{\mu m}^{p}C_{\mu-m}^{m+\frac{p-2}{2}}(\cos\chi)\sin^{m}\chi\Xi_{M}^{m}(\boldsymbol{\vec{\xi}}\;)e^{i\lambda\phi},\nonumber \\
	&  & \mu\geq m\geq0;
	\end{eqnarray}
	\item In the case of $p=q=2$:
	\begin{equation}
	\Xi_{\lambda\mu}^{(\sigma,\varepsilon)}(\phi;\chi)=\frac{1}{2\pi}e^{i(\lambda\phi+\mu\chi)}.
	\end{equation}
\end{enumerate}

In formulas (17)-(18) the following notations are adopted:

$\lambda,\mu,l,m$ -- are integers;

$L,M$--multi- indices which are non-negative integers;

£$\Xi_{L}^{l}(\boldsymbol{\vec{\eta}}),\:\Xi_{M}^{m}(\boldsymbol{\vec{\xi}}\;)$ -- are the elements of canonical basis of $SO(p-1)$ and $SO(q-1)$ groups, respectively, [3];

$a_{\lambda l\mu m}^{pq}$, $a_{\mu m}^{p}$ -- are normalization multipliers which are selected in such a way that the elements of canonical basis (17)-(19) form an orthonormalized system with respect fo the scalar product (7):
\begin{eqnarray}
a_{\lambda l\mu m}^{pq} & = & \frac{\Gamma\left(l+\frac{q-2}{2}\right)\Gamma\left(m+\frac{p-2}{2}\right)}{\pi2^{4-l-m-(p+q)/2}}\times\nonumber \\
& \times & \sqrt{\frac{(\lambda-l)!(\mu-m)!(2\lambda+q-2)(2\mu+p-2)}{\Gamma(\lambda+l+q-2)\Gamma(\mu+m+p-2)}},
\end{eqnarray}
\begin{equation}
a_{\mu m}^{p}=\frac{\Gamma\left(m+\frac{p-2}{2}\right)}{\pi}\sqrt{2^{p+2m-5}\frac{(\mu-m)!(2\mu+p-2)}{\Gamma(\mu+m+p-2)}}.
\end{equation}
Furthermore, there is a restriction by the parity of the representation $\varepsilon$:

\begin{equation}
\lambda+\mu=\varepsilon\left(\mathrm{mod}\,2\right)\label{eq:310}
\end{equation}

The integral representations for matrix elements of the operator of hyperbolic rotations on the plane $(k_{1},k_{p+q})$	in canonical basis, particularly for associated functions, can be written using formulas (5)-(6) and (11)-(13). Assuming the notations	for	associated	functions,	we	obtain the following integral	representations using group-theoretical methods:
\begin{enumerate}
\item[1.] In the case of $p\geq q\geq3$:
\begin{eqnarray}
P_{\sigma\lambda\mu}^{[p,q]}(\alpha) & = & a_{\lambda\mu}^{pq}\intop_{-1}^{+1}\intop_{-1}^{+1}\Lambda^{\sigma/2}C_{\mu}^{\frac{p-2}{2}}(x)C_{\lambda}^{\frac{q-2}{2}}(y)\times\nonumber \\
& \times & \left(1-x^{2}\right)^{\frac{p-3}{2}}\left(1-y^{2}\right)^{\frac{q-3}{2}}dxdy;
\end{eqnarray}
\item[2.] In the case of $p\geq3,\;q=2$:
\begin{equation}
P_{\sigma\lambda\mu}^{[p,2]}(\alpha)=a_{\mu}^{p}\intop_{0}^{2\pi}\intop_{-1}^{+1}\Lambda^{\sigma/2}C_{\mu}^{\frac{p-2}{2}}(x)e^{-i\lambda\phi}\left(1-x^{2}\right)^{\frac{p-3}{2}}d\phi dx;
\end{equation}
\item[3.] In the case of $p=q=2$:
\begin{equation}
P_{\sigma\lambda\mu}^{[2,2]}(\alpha)=\frac{1}{4\pi^{2}}\intop_{0}^{2\pi}\,\intop_{0}^{2\pi}\Lambda^{\sigma/2}e^{-i(\lambda\phi+\mu\chi)}d\phi d\chi.
\end{equation}	
\end{enumerate}
Here we used the following notations:
\begin{eqnarray}
a_{\mu}^{p} & = & \Gamma\left(\frac{p-2}{2}\right)\sqrt{\frac{2^{p-6}\mu!(2\mu+p-2)\Gamma\left(\frac{p}{2}\right)}{\pi^{7/2}\Gamma(\mu+p-2)\Gamma\left(\frac{p-1}{2}\right)}}\quad ,\nonumber\\
a_{\lambda\mu}^{pq} & = & \Gamma\left(\frac{p-2}{2}\right)\Gamma\left(\frac{q-2}{2}\right)\sqrt{\frac{2^{p+q-8}\Gamma\left(\frac{p}{2}\right)\Gamma\left(\frac{q}{2}\right)}{\pi^{3}\Gamma\left(\frac{p-1}{2}\right)\Gamma\left(\frac{q-1}{2}\right)}}\times\nonumber\\
& \times & \sqrt{\frac{\lambda!\mu!(2\mu+p-2)(2\lambda+q-2)}{\Gamma(\lambda+q-2)\Gamma(\mu+p-2)}}\ .
\end{eqnarray}

Remember that $\Lambda^{\sigma/2}$--function is defined by (11).

It is to be noted that the associated functions exist only for even $(\varepsilon=0)$ representation as in the case of zonal functions. This follows directly from their definition. Particularly, in order to account for the condition (22), we must assume that:
\begin{equation}
\lambda=\nu+2r,\quad\mu=\nu+2s,\quad\nu=0,1,
\end{equation}
where $r,s$ -- are positive integers and introduce the following function:
\begin{equation}
\text{\ensuremath{\mathbb{P}}}_{\sigma rs}^{pq\nu}(\alpha)=P_{\sigma\lambda\mu}^{[p,q]}(\alpha).
\end{equation}
The restrictions on the possible values of $r,s$ are the same as the restrictions for $\lambda,\mu$.
\section{Associated functions and Horn's series}
The main formulas used for the calculation of zonal functions of $SO_{0}(p,q)$ groups are the integral representations (8)-(10) and the Taylor expansions for the $\Lambda^{\sigma/2}$-function given as:

\begin{eqnarray*}
	\text{\ensuremath{\mathbb{P}}}_{\sigma rs}^{pq\nu}(\alpha) & = & \frac{(-1)^{s+r+\nu}}{\cosh\alpha}A_{1}^{pq}A_{2}^{pq}\sum_{l=\max(s,r)}^{\infty}\frac{l!\left(\nu+\frac{1}{2}\right)_{l}}{(l-s)!(l-r)!}(\tanh\alpha)^{2l+\nu}\times
\end{eqnarray*}
\begin{equation}
\times F_{2}\left(s+r+\nu-\frac{\sigma}{2},s-l,r-l;2s+\nu+\frac{p}{2},2r+\nu+\frac{q}{2};1,1\right),
\end{equation}
where
\begin{eqnarray*}
	A_{1}^{pq} & = & \frac{2^{\frac{p+q}{2}+\nu-3}(-\sigma/2)_{s+r+\nu}}{\Gamma\left(2s+\nu+\frac{p}{2}\right)\Gamma\left(2r+\nu+\frac{q}{2}\right)}\sqrt{\frac{\pi\Gamma(p/2)\Gamma(q/2)}{\Gamma\left(\frac{p-1}{2}\right)\Gamma\left(\frac{q-1}{2}\right)}}\times\\
	& \times & \sqrt{\frac{\Gamma(2s+\nu+p-1)\Gamma(2r+\nu+q-1)}{(2s+\nu)!(2r+\nu)!}},
\end{eqnarray*}
\[
A_{2}^{pq}=\sqrt{\frac{\left(2s+\nu+\frac{p-2}{2}\right)\left(2r+\nu+\frac{q-2}{2}\right)}{(2s+\nu+p-2)(2r+\nu+q-2)}}.
\]
Eq.(29) for associated functions of $SO(p,q)$ groups can be rewritten more compactly using generalized hypergeometric functions of the two variables, (Appendix, Eq.51):
\begin{eqnarray*}
	\text{\ensuremath{\mathbb{P}}}_{\sigma rs}^{pq\nu}(\alpha) & = & \frac{(2s+\nu)!\left(\frac{2-\sigma-q}{2}\right)_{s-r}}{(s-r)!4^{s}\cosh\alpha}A_{1}^{pq}A_{2}^{pq}(\tanh\alpha)^{2s+\nu}\times
\end{eqnarray*}
\begin{equation}
\times\,{}_{5}F_{3}\left[\begin{array}{c|cc|c}
11 & s+1, & s+\nu+\frac{1}{2}\\
10 & s+r+\nu-\frac{\sigma}{2}&\\
01 & \frac{q+\sigma}{2}, & s-r+\frac{p+\sigma}{2} & \tanh^{2}\alpha,\:\tanh^{2}\alpha\\\cline{1-3}
11 & 2s+\nu+p/2, & s+r+\nu+q/2\\
01 & 1+s-r&
\end{array}\right]
\end{equation}
Formula (30) is valid for $s\geq r$. For $s\leq r$ the expression for associated function is derived from (29) using the rearrangement:
\[ \begin{bmatrix}p & q & r & s\\
q & p & s & r
\end{bmatrix}\]

From this, one can attain the symmetry property of the associated function for the groups $SO(p,q)$:
\begin{equation}
\text{\ensuremath{\mathbb{P}}}_{\sigma rs}^{pq\nu}(\alpha)=\text{\ensuremath{\mathbb{P}}}_{\sigma sr}^{q\text{p}\nu}(\alpha)\text{.}\label{eq:320}
\end{equation}
It is also important to note that formulas (30)-(32) are true for all $SO_{0}(p,q),\;p\geq2,\;q\geq2$ groups.

Assuming $s=r=\nu=0$ in (29)-(30), the expressions for zonal functions of $SO_{0}(p,q),\;p\geq2,\;q\geq1$ groups can be derived as:
\begin{equation}
Z_{\sigma}^{[p,q]}(\alpha)=\frac{1}{\cosh\alpha}\sum_{l=0}^{\infty}\frac{\left(\frac{1}{2}\right)_{l}}{l!}F_{2}\left(-\frac{\sigma}{2},-l,-l;\frac{p}{2},\frac{q}{2};1,1\right)\left(\tanh\alpha\right)^{l/2};\label{eq:321}
\end{equation}
\begin{equation}
Z_{\sigma}^{[p,q]}(\alpha)=\frac{1}{\cosh\alpha}{}_{5}F_{3}\left[\begin{array}{c|cc|c}
11 & 1, & 1/2\\
10 & -\sigma/2&\\
01 & \frac{q+\sigma}{2}, & \frac{p+\sigma}{2} & \tanh^{2}\alpha,\:\tanh^{2}\alpha\\\cline{1-3}
11 & p/2, & q/2\\
01 & 1&
\end{array}\right]
\end{equation}

Notice that the expressions (15)-(16) and (33) for zonal functions are equivalent as they are both derived from (32) depending on the selection of representation of the Appell's function a second kind, which is used for the summation. The advantages of E(32)-(33) involve the fact that symmetry between $p$ and $q$ is evident.

The following expansions are obtained from (23) - (25) for the function $\Lambda^{\sigma/2}$:
\begin{eqnarray}
\Lambda^{\sigma/2} & = & \frac{\pi\Gamma\left(\frac{p-1}{2}\right)\Gamma\left(\frac{q-1}{2}\right)}{\Gamma(p/2)\Gamma(q/2)}\sum_{\lambda\mu}a_{\lambda\mu}^{pq}P_{\sigma\lambda\mu}^{[p,q]}(\alpha)C_{\mu}^{\frac{p-2}{2}}(x)C_{\lambda}^{\frac{q-2}{2}}(y),\label{eq:323}\\
\Lambda^{\sigma/2} & = & \frac{2\pi^{3/2}\Gamma\left(\frac{p-1}{2}\right)}{\Gamma(p/2)}\sum_{\lambda\mu}a_{\mu}^{p}P_{\sigma\lambda\mu}^{[p,2]}(\alpha)C_{\mu}^{\frac{p-2}{2}}(x)e^{i\lambda\phi},\label{eq:324}\\
\Lambda^{\sigma/2} & = & \sum_{\lambda\mu}P_{\sigma\lambda\mu}^{[2,2]}(\alpha)e^{i(\lambda\phi+\mu\chi)},\label{eq:325}
\end{eqnarray}
where summations are performed over all admissible values $\lambda,\mu$, which satisfy the restriction $\lambda+\mu=0\:(\mathrm{mod}\,2)$.

Expansions (34)-(36) allow supplementary functional relationships to be obtained for the spherical functions for $SO_{0}(p,q)$ groups, and these may be useful in the solutions of many problems. It should be noted that the convergence of the series in (34)-(36) must be understood in the sense of convergence of distributions over Fr\`{e}chet spaces [12]: 
\[\textbf{$C^{\infty}\left(\left[-1,1\right]\times\left[-1,1\right]\right),\;C^{\infty}\left(\left[-1,1\right]\times\left[0,2\pi\right]\right),\;C^{\infty}\left(\left[0,2\pi\right]\times\left[0,2\pi\right]\right)$,} \]
respectively\footnote{In this case the ends of interval $\left[0,2\pi\right]$ must be identified.}.

To conclude, it is to be noted that we have studied the functions of $SO_{0}(p,q)$ groups for the representations of the continuous principal series. The study of representations of a complementary series is reduced to the analytic continuation in the $\sigma$. 

The theorem on the transformation of derivative distributions, concentrated on smooth surfaces with respect to infinite-dimensional Lie groups $C^{\infty}\left(\mathbb{R^{\mathit{n}}},GL(n)\right)$ plays an important role in the study of representations of discrete and exceptional series. This theorem is of independent significance, so we give the proof in the next section.

\section{Transformation of derivatives of distributions\\concentrated on smooth surfaces}
The problem of transformation of generalized functions, concentrated on smooth surfaces [13], and their derivatives arise in many problems of theoretical and mathematical physics and in particular, in the theory of representations of Lie groups. In this section the problem is solved for the infinite-dimensional Lie group $C^{\infty}\left(\mathbb{R^{\mathit{n}}},GL(n)\right)$.

Suppose that $(n-k)$ -- dimensional surface $\mathbb{S}$ in $\mathbb{R}^{\mathit{n}}$ is given by two different system of equations:
\begin{eqnarray}
\mathcal{P}_{i}(x) & = & 0,\quad\quad\mathcal{P}_{i}(\cdot)\in C^{\infty}\left(\mathbb{R^{\mathit{n}}}\right),\nonumber \\
\mathcal{Q}_{i}(x) & = & 0,\quad\quad\mathcal{Q}_{i}(\cdot)\in C^{\infty}\left(\mathbb{R^{\mathit{n}}}\right);\quad i=1,2,\ldots,k.
\end{eqnarray}
Here functions $\mathcal{P}_{i}(\cdot)$ and $\mathcal{Q}_{i}(\cdot)$ are interconnected:
\begin{equation}
\mathcal{Q}_{i}(x)=\sum_{j}\mathcal{P}_{j}(x)\alpha_{ji}(x);\qquad\mathcal{P}_{i}(x)=\sum_{j}\mathcal{Q}_{j}(x)\beta_{ji}(x),\label{eq:Q23}
\end{equation}
where
\begin{eqnarray}
\sum_{i}\alpha_{ji}(x)\beta_{ik}(x) & = & \delta_{jk};\qquad\sum_{i}\beta_{ji}(x)\alpha_{ik}(x)=\delta_{jk},\nonumber \\
\alpha_{ij}(\cdot)\in C^{\infty}\left(\mathbb{R^{\mathit{n}}}\right), &  & \beta_{ij}\in C^{\infty}\left(\mathbb{R^{\mathit{n}}}\right).\label{eq:Q25}
\end{eqnarray}
The last conditions imply, in particular, non-singularity of matrices $\alpha\equiv\Vert\alpha_{ij}\Vert$ and $\beta\equiv\Vert\beta_{ij}\Vert$, that is the following conditions hold true:
\begin{eqnarray}
\det\alpha\cdot\det\beta & = & 1,\qquad\det\alpha\neq0,\quad\det\beta\neq0.\label{eq:Q26}
\end{eqnarray}
Obviously matrices $\Vert\alpha\Vert$ and $\Vert\beta\Vert$ are elements of  the infinite group Lie $C^{\infty}\left(\mathbb{R^{\mathit{n}}},GL(n)\right)$.

In addition, assume that the family of surfaces $\mathcal{P}_{i}(x)=\mathrm{const}$ and $\mathcal{Q}_{i}(x)=\mathrm{const}$ form a correct grid, [13]. Under these assumptions there exists the following relationship between $\delta$-functions, concentrated on the surface $\mathbb{S}$ and corresponding equations (37)-(38), [13]:
\begin{equation}
\delta\left(\mathcal{Q}_{1},\ldots,\mathcal{Q}_{k}\right)=\frac{1}{\det\alpha}\delta\left(\mathcal{P}_{1},\ldots,\mathcal{P}_{k}\right).
\end{equation}
Using (41) this equality can be rewritten as:
\begin{equation}
\delta\left(\mathcal{Q}_{1},\ldots,\mathcal{Q}_{k}\right)=\det\beta\,\cdot\delta\left(\mathcal{P}_{1},\ldots,\mathcal{P}_{k}\right).
\end{equation}
Our aim is to obtain an analogous relation for derivatives of $\delta$-functions. 

In the rest of this section double repeated indexes will mean summation over all possible values of indexes (from 1 to $n$ for indexes of coordinates of $x$, from 1 to $k$ for indexes of variables $\mathcal{P}$, and $\mathcal{Q}$). We also use the following notations:
\begin{eqnarray*}
	\delta_{ijl\cdots}(\mathcal{Q}) & = & \frac{\partial}{\partial\mathcal{Q}_{i}\partial\mathcal{Q}_{j}\partial\mathcal{Q}_{l}\cdots}\delta\left(\mathcal{Q}\right),\\
	\delta_{ijl\cdots}(P) & = & \frac{\partial}{\partial{P}_{i}\partial{P}_{j}\partial{P}_{l}\cdots}\delta\left(\mathcal{{P}}\right).
\end{eqnarray*}
\begin{theorem}
	The following holds true:
\begin{equation}	
	\delta_{i_{1}\cdots i_{s}}(\mathcal{Q})=\left(\det\beta\right)\,\,\cdot\sum_{j_{1}\cdots j_{s}}\,\beta_{i_{1}j_{1}}\cdots\beta_{i_{s}j_{s}}\delta_{j_{1}\cdots j_{s}}(P).
\end{equation}
\end{theorem}
\begin{proof}
We prove this formula using the mathematical induction method. Assume that (43) is true. Apply differentiation to it with respect to  $x_{\mu}$:
\[
\frac{\partial}{\partial x_{\mu}}\delta_{i_{1}\cdots i_{s}}(\mathcal{Q})=\frac{\partial\det\beta}{\partial x_{\mu}}\,\beta_{i_{1}j_{1}}\cdots\beta_{i_{s}j_{s}}\delta_{j_{1}\cdots j_{s}}(P)+\left(\det\beta\right)\,\beta_{i_{1}j_{1},\mu}\cdots\beta_{i_{s}j_{s}}\delta_{j_{1}\cdots j_{s}}(P)+
\]
\begin{equation}
\cdots+\left(\det\beta\right)\,\beta_{i_{1}j_{1}}\cdots\beta_{i_{s}j_{s},\mu}\delta_{j_{1}\cdots j_{s}}(P)+\left(\det\beta\right)\,\beta_{i_{1}j_{1}}\cdots\beta_{i_{s}j_{s}}\delta_{j_{1}\cdots j_{s}}(P)P_{j,\mu},
\end{equation}
where
\[
\beta_{ij,\mu}\equiv\frac{\partial\beta_{ij}}{\partial x_{\mu}},\qquad{P}_{j,\mu}\equiv\frac{\partial{P}_{j}}{\partial x_{\mu}}.
\]
The derivative  $\frac{\partial\det\beta}{\partial x_{\mu}}$ is easily calculated:
\begin{equation}
\frac{\partial\det\beta}{\partial x_{\mu}}=\beta_{ij,\mu}\frac{\partial\det\beta}{\partial\beta_{ij}}=\beta_{ij,\mu}\alpha_{ji}\cdot\det\beta.
\end{equation}
Further, applying differentiation to $\mathcal{Q}_{l}\delta(\mathcal{Q})=0$ sequentially with respect to $\mathcal{Q}_{i_{1}}\cdots\mathcal{Q}_{i_{s+1}}$ we get:
\begin{equation}
\sum_{CP(i_{1}\cdots i_{s}i)}\delta_{li}\delta_{i_{1}\cdots i_{s}}(\mathcal{Q})+\mathcal{Q}_{l}\delta_{i_{1}\cdots i_{s}i}(\mathcal{Q})\label{eq:Q32}
\end{equation}
Here summation is done over all $(s+1)$ -- cyclic interchanging of indexes $i_{1}\cdots i_{s}i$. 

On the other side it follows from (38)-(42) that:
\begin{eqnarray*}
	\frac{\partial}{\partial x_{\mu}}\delta_{i_{1}\cdots i_{s}}(\mathcal{Q}) & = & \mathcal{Q}_{i,\mu}\delta_{i_{1}\cdots i_{s}i}(\mathcal{Q})=\\
	& = & \alpha_{ji,\mu}P_{j}\delta_{i_{1}\cdots i_{s}i}(\mathcal{Q})+\alpha_{ji}P_{j,\mu}\delta_{i_{1}\cdots i_{s}i}(\mathcal{Q})=\\
	& = & \alpha_{ji,\mu}\beta_{lj}\mathcal{Q}_{l}\delta_{i_{1}\cdots i_{s}i}(\mathcal{Q})+\alpha_{ji}P_{j,\mu}\delta_{i_{1}\cdots i_{s}i}(\mathcal{Q})=\\
	& = & -\alpha_{ji,\mu}\beta_{lj}\sum_{CP(ii_{1}\cdots i_{s})}\delta_{li}\delta_{i_{1}\cdots i_{s}}(\mathcal{Q})+\alpha_{ji}P_{j,\mu}\delta_{i_{1}\cdots i_{s}i}(\mathcal{Q})=
\end{eqnarray*}
\begin{eqnarray*}
	& = & \alpha_{ji,\mu}\det\beta\sum_{CP(ii_{1}\cdots i_{s})}\beta_{ij}\beta_{i_{1}j_{1}}\cdots\beta_{i_{s}j_{s}}\delta_{j_{1}\cdots j_{s}}(P)+\alpha_{ji}P_{j,\mu}\delta_{i_{1}\cdots i_{s}i}(\mathcal{Q})=\\
	& = & \alpha_{ji}\det\beta\sum_{CP(i_{1}\cdots i_{s}i)}\beta_{ij,\mu}\beta_{i_{1}j_{1}}\cdots\beta_{i_{s}j_{s}}\delta_{j_{1}\cdots j_{s}}(P)+\alpha_{ji}P_{j,\mu}\delta_{i_{1}\cdots i_{s}i}(\mathcal{Q}),
\end{eqnarray*}
where functions $\mathcal{P}$ and $\mathcal{Q}$ are interconnected.

Comparing the last expression with (44) and taking into account (45) we obtain:
\[
\left(\det\beta\right)\,P_{j,\mu}\beta_{i_{1}j_{1}}\cdots\beta_{i_{s}j_{s}}\delta_{j_{1}\cdots j_{s}j}(P)=\alpha_{ji}P_{j,\mu}\delta_{i_{1}\cdots i_{s}i}(\mathcal{Q}).
\]
From here using (38) we have:
\[
\delta_{i_{1}\cdots i_{s}i}(\mathcal{Q})=\left(\det\beta\right)\,\beta_{i_{1}j_{1}}\cdots\beta_{i_{s}j_{s}}\delta_{j_{1}\cdots j_{s}j}(P).
\]
It is easy to see that the last equality coincides with the formula (42) for $s+1$.

On the other side the formula (43) is true for $s=0$ as this case coincides with (42). Thus, the proof of the formula (43) is completed.
\end{proof}
We transform the main result (44) to the form suitable for applications. 

Introduce new notations:
\begin{eqnarray*}
	\delta^{\left(p_{1}\cdots p_{k}\right)}(P) & = & \frac{\partial}{\partial{P}_{\mathit{1}}^{p_{1}}\cdots\partial{P}_{k}^{p_{k}}}\delta(P),\\
	\left|p\right| & = & p_{1}+\cdots+p_{k}.
\end{eqnarray*}
Apply the formula (43) to the derivative of the $\delta$-function:
\begin{eqnarray*}
	\delta^{\left(q_{1}\cdots q_{k}\right)}(\mathcal{Q}) & = & \left(\det\beta\right)\,\beta_{\mathit{1}j_{1},1}\cdots\beta_{1j_{1},q_{1}}\beta_{\mathit{2}j_{2},1}\cdots\beta_{2j_{2},q_{2}}\times\\
	& \times & \beta_{\mathit{k}j_{k},1}\cdots\beta_{kj_{k},q_{k}}\delta_{j_{1}1\cdots j_{1}q_{1},\cdots,j_{k}1\cdots j_{k}q_{k}}(P).
\end{eqnarray*}
Next it is necessary to group similar terms on the right hand side of this equality that is to count the number of terms of the form:
\[
\left(\det\beta\right)\left(\prod_{i,j=1}^{k}\left(\beta_{ij}\right)^{r_{ij}}\right)\delta^{\left(p_{1}\cdots p_{k}\right)}(P).
\]
It is easy to see that for fixed $i$, the number of such terms coincides with the number of packing of $r_{i1}+\cdots+r_{ik}$ objects of $k$ different types into $q_{i}$ boxes [14] that is $q_{i}/\left(r_{i1}!r_{i2}!\cdots r_{ik}!\right)$.
Then applying the rule of differentiation k times [14] we get the final result:
\begin{equation}
\delta^{\left(q_{1}\cdots q_{k}\right)}(\mathcal{Q})=\left(\det\beta\right)\sum_{r_{ij}}\prod_{i=1}^{k}q_{i}!\prod_{j=1}^{k}\frac{\left(\beta_{ij}\right)^{r_{ij}}}{r_{ij}!}\delta^{\left(p_{1}\cdots p_{k}\right)}(P),
\end{equation}
where
\[
\sum_{j=1}^{k}r_{ij}=q_{i},\qquad\sum_{i=1}^{k}r_{ij}=p_{j}.
\]
Note that it follows from these formulas that:
\[
\sum_{i=1}^{k}q_{i}=\sum_{j=1}^{k}p_{j}.
\]
\section{Zonal functions of groups $SO(4,1)$, $SO(3,2)$ and $SO(4,2)$}
Since de Sitter groups and the group of conformal invariance play an important role in cosmology and the theory of elementary particles, in this section we present explicit expressions for zonal spherical functions of these groups.

Directly from the formula (14) we get the explicit expression for the zonal spherical function of the group $SO(4,1)$: 

\begin{equation}
Z_{\sigma}^{[4,1]}(\alpha)=\,_{2}F_{1}\left(-\frac{\sigma}{2},\frac{\:1-\sigma}{2};\:2;\:\tanh^{2}\alpha\right)(\cosh\alpha)^{\sigma}.
\end{equation}

Directly from the formula (32) we obtain the explicit expressions for the zonal spherical functions of groups $SO(3,2)$ and $SO(4,2)$: 

\begin{eqnarray}
Z_{\sigma}^{[3,2]}(\alpha) & = & \frac{1}{\cosh\alpha}\,\sum_{l=0}^{\infty}\frac{\left(\frac{1}{2}\right)_{l}}{l!}F_{2}\left(-\frac{\sigma}{2},-l,-l;\frac{3}{2},1;1,1\right)\left(\tanh\alpha\right)^{l/2}=\nonumber \\
& = & \frac{1}{\cosh\alpha}\,\sum_{l=0}^{\infty}\,\frac{\left(\frac{1}{2}\right)_{l}\left(\frac{3}{2}\right)_{l}\left(\frac{3+\sigma}{2}\right)_{l}\left(\frac{\sigma}{2}+1\right)_{l}}{\left(l!\right)^{2}}\times\nonumber \\
& \times & _{3}F_{2}\begin{pmatrix}-\frac{\sigma}{2}, & -l, & -l;\\
&  &  & 1\\
-\frac{\sigma+1}{2}-l, & -\frac{\sigma}{2}-l;
\end{pmatrix}\left(\tanh\alpha\right)^{l/2}
\end{eqnarray}

\begin{eqnarray}
Z_{\sigma}^{[4,2]}(\alpha) & = & \frac{1}{\cosh\alpha}\,\sum_{l=0}^{\infty}\frac{\left(\frac{1}{2}\right)_{l}}{l!}F_{2}\left(-\frac{\sigma}{2},-l,-l;2,1;1,1\right)\left(\tanh\alpha\right)^{l/2}=\nonumber \\
& = & \frac{1}{\cosh\alpha}\,\sum_{l=0}^{\infty}\frac{\left(\frac{1}{2}\right)_{l}\left(\frac{\sigma}{2}+2\right)_{l}\left(\frac{\sigma}{2}+1\right)_{l}}{\left(l!\right)^{2}\Gamma\left(l+2\right)}\times\nonumber \\
& \times & _{3}F_{2}\begin{pmatrix}-\frac{\sigma}{2}, & -l, & -l;\\
&  &  & 1\\
-\frac{\sigma}{2}-l-1, & -\frac{\sigma}{2}-l;
\end{pmatrix}\left(\tanh\alpha\right)^{l/2}
\end{eqnarray}

In conclusion, I would like to thank Prof. E.Veliev  and Prof. A.Bagirov for their attention to this work and for discussion of results.

\section*{Appendix. The generalized hypergeometric Horn's series}

The generalized hypergeometric Horn's series with $r$ variables is defined as follows, [15]-[16]:
\[
_{p}F_{q}\left[\begin {array}{ccc|c|c}
u_{\alpha_{1}1}, & \ldots, & u_{\alpha_{1}r} & \left\{ a_{\alpha_{1}}\right\} \\
& \ddots &  & \vdots\\
u_{\alpha_{s}1}, & \ldots, & u_{\alpha_{s}r} & \left\{ a_{\alpha_{s}}\right\}  & x_{1},\ldots x_{r}\\\cline{1-4}
v_{\beta_{1}1}, &\ldots, & v_{\beta_{1}r} & \left\{ b_{\beta_{1}}\right\} \\
& \ddots &  & \vdots\\
v_{\beta_{t}1}, & \ldots, & v_{\beta_{t}r} & \left\{ b_{\beta_{1}}\right\} 
\end{array}\right]=
\]
\begin{equation}
=\sum_{n_{1}\ldots n_{r}=0}^{\infty}\,\frac{\underset{\alpha=1}{\overset{p}{\prod}}(a_{\alpha},\underset{j=1}{\overset{r}{\sum}}u_{\alpha j}n_{j})}{\underset{\beta=1}{\overset{q}{\prod}}(b_{\beta},\underset{j=1}{\overset{r}{\sum}}v_{\beta j}n_{j})}\cdot\prod_{i=1}^{r}\frac{x_{i}^{n_{i}}}{n_{i}!}.
\end{equation}

Below we present notations used in (51) and satisfying conditions:

\begin{eqnarray*}
	\sum_{\alpha=1}^{p}u_{\alpha j} & = & \sum_{\beta=1}^{q}v_{\beta j}+1,\\
	(\lambda,n) & = & \frac{\Gamma(\lambda+n)}{\Gamma(\lambda)}=(\lambda)_{n},\qquad(\lambda,0)=1.
\end{eqnarray*}
where the multiples in the numerator (denominator) of (51), corresponding to the same values of $u_{\alpha j}$ (respectively, $v_{\beta j}$), $1\leq j\leq r$, are unified in the same row of (51), i.e., the set of the same values of $\alpha,\quad1\leq\alpha\leq p,$ (respectively, $\beta,\quad1\leq\beta\leq q$) are decomposed to the subsets $\alpha_{1},\ldots,\alpha_{s}$ (respectively, $\beta_{1},\ldots,\beta_{t}$) with the same values $u_{\alpha j}$ (respectively, $v_{\beta j}$), $1\leq j\leq r$.



\bigskip

\begin{thebibliography}{10}

\bibitem{} A.O. Barut, R. Raczka, (1977), Theory of Group Representations and Applications, Warszawa.

\bibitem{} Theoretico-group methods in physics. Trudi mejd. semin. v Zveinigorode. - Moscow, "Nauka", v. 1-2 (1983) pp. 499-505. (in Russian).

\bibitem{} Vilenkin N.Ya. Special functions and the theory of groups representations.- Moscow, "Nauka", (1991). (in Russian).

\bibitem{}  A.U. Klimik, Matrix elements and Clebsh-Gordon coefficients of group representations.-Kiev, "Naukovo dumka", (1979). (in Russian).

\bibitem{}[5]  C.S. Herz, Ann. of Math., 61, N3 (1955), pp. 474-523.

\bibitem{} B.A. Rajabov, Proc. Nat. Acad. Sci., Azerbaijan, 44, N5 (1988), pp. 48-52; (in Russian). 

\bibitem{} B.A. Rajabov, Proc. Nat. Acad. Sci., Azerbaijan, 44, N7 (1988), pp.21-25; (in Russian).

\bibitem{} B.A. Rajabov, Tr. J. of Physics, 18 (1994), 590 - 599.

\bibitem{} Rajabov B.A.,(2003) arXiv:math-ph/0302051, p.6.

\bibitem{} Rajabov B.A.,(2003) arXiv:math-ph/0303058, p.10.

\bibitem{} Slater L.J. Generalized hypergeometric functions, (1966), Cambridge.

\bibitem{} Treves F.,Topological Vector Spaces, Distributions and Kernels, (1967), Purdue University, Lafayette, Indiana.

\bibitem{} I.M. Gelfand, G.E. Shilov, Generalized functions and operations on them, v.1,(1959) (in Russian).

\bibitem{} F. Riordan, An introduction to combinatorial analysis, (1958), London, Chapman and Hall. 

\bibitem{} J. Horn, Math. Ann., 34, 544 (1889).

\bibitem{} A.N. Leznov, M. V. Saveliev, Preprint 72-3, Serpukhov, (1972).

\end{thebibliography}
\end{document}